\documentclass[12pt,american]{article}
\usepackage{color}
\usepackage{babel}
\usepackage{units}
\usepackage{amsmath}
\usepackage{amsthm}
\usepackage{amssymb}
\usepackage{mathrsfs}
\usepackage[unicode=true,
 bookmarks=true,bookmarksnumbered=true,bookmarksopen=false,
 breaklinks=true,pdfborder={0 0 1},backref=false,colorlinks=true]
 {hyperref}
\hypersetup{pdftitle={Fractional kinetic in spatial ecological model},
 pdfauthor={{Jos{\'e}Lu{\'\i}sdaSilva},
 pdfsubject={General Fractional calculus and asymptotic of traveling waves},
 pdfkeywords={Bolker-Pacala model, General fractional derivative, Distributed order derivative, Subordination principle},
 },linkcolor=blue,citecolor=blue,urlcolor=blue,filecolor=blue}

\makeatletter

\theoremstyle{plain}
\newtheorem{thm}{}[section]
\theoremstyle{definition}
\newtheorem{example}[thm]{}\theoremstyle{plain}
\newtheorem{prop}[thm]{}\theoremstyle{plain}
\newtheorem{lem}[thm]{}\theoremstyle{remark}
\newtheorem{rem}[thm]{}\theoremstyle{definition}
\newtheorem{defn}[thm]{}

\usepackage{pdfsync}
\numberwithin{equation}{section}

\makeatother

\begin{document}
\title{Random Time Change and Related Evolution Equations\\
{\large{}Time Asymptotic Behavior}}
\author{\textbf{Anatoly N. Kochubei}\\
 Institute of Mathematics,\\
 National Academy of Sciences of Ukraine, \\
 Tereshchenkivska 3, \\
 Kyiv, 01004 Ukraine\\
 Email: kochubei@imath.kiev.ua\and\textbf{Yuri Kondratiev}\\
 Department of Mathematics, University of Bielefeld, \\
 D-33615 Bielefeld, Germany,\\
 Dragomanov University, Kiev, Ukraine\\
 Email: kondrat@math.uni-bielefeld.de\and\textbf{Jos{\'e} Lu{\'i}s
da Silva},\\
 CIMA, University of Madeira, Campus da Penteada,\\
 9020-105 Funchal, Portugal.\\
 Email: joses@staff.uma.pt}
\date{\today}
\maketitle
\begin{abstract}
In this paper we investigate the long time behavior of solutions to
fractional in time evolution equations which appear as results of
random time changes in Markov processes. We consider inverse subordinators
as random times and use the subordination principle for the solutions
to forward Kolmogorov equations. The class of subordinators for which
asymptotic analysis may be realized is described.
\end{abstract}

\section{Introduction}

We start with a brief description of our framework. Our presentation
will be rather informal. For necessary technical conditions and details
we refer to the main body of this paper.

Let $\{X_{t},t\geq0;P_{x},x\in E\}$ be a strong Markov process in
a phase space $E$. Denote $T_{t}$ its transition semigroup (in a
proper Banach space) and $A$ the generator of this semigroup. Let
$S_{t},t\geq0$ be a subordinator (i.e., a non-decreasing real-valued
L{\'e}vy process) with $S_{0}=0$ and Laplace exponent $\Phi$: 
\[
\mathbb{E}[e^{-\lambda S_{t}}]=e^{-t\Phi(\lambda)}\;\;t,\lambda>0.
\]
We assume that $S_{t}$ is independent of $X_{t}$.

Denote by $E_{t},t>0$ the inverse subordinator and introduce the
time changed process $Y_{t}=X_{E_{t}}$. We are interested in the
time evolution 
\[
u(x,t)=\mathbb{E}^{x}[f(Y_{t})]
\]
for a given initial data $f$. As it was pointed out in several works,
see e.g. \cite{Toaldo2015}, \cite{Chen2017}, $u(x,t)$ is the unique
strong solution (in some proper sense) to the following Cauchy problem
\[
\mathbb{D}_{t}^{(k)}u(x,t)=Au(x,t)\;\;u(x,0)=f(x).
\]
Here we have a generalized fractional derivative (see \cite{Kochubei11})
\[
\mathbb{D}_{t}^{(k)}\phi(t)=\frac{d}{dt}\int_{0}^{t}k(t-s)(\phi(s)-\phi(0))ds
\]
with a kernel $k$ uniquely defined by $\Phi$.

Let $u_{0}(x,t)$ be the solution to a similar Cauchy problem but
with ordinary time derivative. In stochastic terminology, it is the
solution to the forward Kolmogorov equation corresponding to the process
$X_{t}$. Under quite general assumptions there is a nice and essentially
obvious relation between these evolutions: 
\[
u(x,t)=\int_{0}^{\infty}u_{0}(x,\tau)G_{t}(\tau)d\tau,
\]
where $G_{t}(\tau)$ is the density of $E_{t}$. Of course, we may
have similar relations for fundamental solutions to the considered
equations, for the backward Kolmogorov equations or time evolutions
of other related quantities.

Having in mind the analysis of the influence of the random time change
on the asymptotic properties of $u(x,t)$, we may hope that the latter
formula gives all necessary technical equipments. Unfortunately, the
situation is essentially more complicated. In fact, the knowledge
about the density $G_{t}(\tau)$ is, in general, very poor. There
are two particular cases in which the asymptotic analysis was already
realized. First of all, it is the situation of the so-called stable
subordinators. Starting with the pioneering works by Meerschaert and
his collaborators, this case was studied in details \cite{BM01,MS2004}.

Another case is related to a scaling property assumed for $\Phi$
\cite{Chen2018}. It is, nevertheless, difficult to give an interpretation
of this scaling assumption in terms of the subordinator.

The aim of this paper is to describe a class of subordinators for
which we may obtain information about the time asymptotic of the generalized
fractional dynamics. We propose two methods for the study of this
problem. In the first approach we use a modified version of the ratio
Tauberian theorem from \cite{Li2007}. This method works under general
assumptions about the integrability in time of the solution $u_{0}(x,t)$.
Actually, under this assumption the asymptotic is determined completely
by the subordinator characteristics.

There is another side of the problem. In many interesting cases the
integrability assumption is not valid. Or, vice versa, we have more
detailed information about the behavior of $u_{0}(x,t)$ which is
much stronger than integrability (e.g., exponential decay). We propose
an alternative approach to such situations based on the Laplace transform
techniques. It gives us the possibility to study solutions without
the integrability property and to see the effects of a stronger decay
of $u_{0}(x,t)$.

Finally, we apply our methods to the study of fractional dynamics
in several particular models: the heat equation, non-local diffusion,
solutions with exponential decays. There we see that the general method
is working perfectly in space dimensions $d\geq3$. But for physically
important dimensions $d=1,2$ we need our alternative approach.

\section{General Fractional Derivative}

\label{sec:GFD}

\subsection{Definitions and Assumptions}

In this section we recall the concept of general fractional derivative
(GFD) associated to a kernel $k$, see \cite{Kochubei11} and references
therein. The basic ingredient of the theory of evolution equations,
\cite{KST2006,Eidelman2004} is to consider, instead of the first
time derivative, the Caputo-Djrbashian fractional derivative of order
$\alpha\in(0,1)$ 
\begin{equation}
\big(\mathbb{D}_{t}^{(\alpha)}u\big)(t)=\frac{d}{dt}\int_{0}^{t}k(t-s)\big(u(s)-u(0)\big)\,ds,\quad t>0,\label{eq:Caputo-derivative}
\end{equation}
where 
\begin{equation}
k(t)=\frac{t^{-\alpha}}{\Gamma(1-\alpha)},\;t>0.\label{eq:kalpha}
\end{equation}
More generally, it is natural to consider differential-convolution
operators 
\begin{equation}
\big(\mathbb{D}_{t}^{(k)}u\big)(t)=\frac{d}{dt}\int_{0}^{t}k(t-s)\big(u(s)-u(0)\big)\,ds,\;t>0,\label{eq:general-derivative}
\end{equation}
where $k\in L_{\mathrm{loc}}^{1}(\mathbb{R}_{+})$ ($\mathbb{R}_{+}:=[0,\infty)$)
is a non-negative kernel. As an example of such an operator, we consider
the distributed order derivative $\mathbb{D}_{t}^{(\mu)}$ corresponding
to 
\begin{equation}
k(t)=\int_{0}^{1}\frac{t^{-\alpha}}{\Gamma(1-\alpha)}\mu(\alpha)\,d\alpha,\quad t>0,\label{eq:distributed-kernel}
\end{equation}
where $\mu(\alpha)$, $0\le\alpha\le1$ is a positive weight function
on $[0,1]$, see \cite{Atanackovic2009,Daftardar-Gejji2008,Hanyga2007,Kochubei2008,Kochubei2008a,Gorenflo2005,Meerschaert2006}.

The class of suitable kernels $k$ we are interested in is such that
the fundamental solution of the corresponding evolution equation \eqref{eq:Evol-eq-1}
in Section~\ref{sec:GFEE}, are probability densities in $L^{\infty}(\mathbb{R}_{+})\cap L^{1}(\mathbb{R}_{+})$.
Therefore, in this paper we make the following assumptions on the
Laplace transform $\mathcal{K}$ of the kernel $k\in L_{\mathrm{loc}}^{1}(\mathbb{R}_{+})$. 
\begin{description}
\item [{(H)}] Let $k\in L_{\mathrm{loc}}^{1}(\mathbb{R}_{+})$ be a non-negative
kernel such that $\int_{0}^{\infty}k(s)\,ds>0$ and its Laplace transform
\begin{equation}
\mathcal{K}(\lambda):=(\mathscr{L}k)(\lambda):=\int_{0}^{\infty}e^{-\lambda t}k(t)\,dt\label{eq:Laplace-k}
\end{equation}
exists for all $\lambda>0$ and $\mathcal{K}$ belongs to the Stieltjes
class (or equivalently, the function $\mathcal{L}(\lambda):=\lambda\mathcal{K}(\lambda)$
belongs to the complete Bernstein function class (see \cite{Schilling12}
for definitions), and 
\begin{equation}
\mathcal{K}(\lambda)\to\infty,\text{ as \ensuremath{\lambda\to0}};\quad\mathcal{K}(\lambda)\to0,\text{ as \ensuremath{\lambda\to\infty}};\label{eq:H1}
\end{equation}
\begin{equation}
\mathcal{L}(\lambda)\to0,\text{ as \ensuremath{\lambda\to0}};\quad\mathcal{L}(\lambda)\to\infty,\text{ as \ensuremath{\lambda\to\infty}}.\label{eq:H2}
\end{equation}
\end{description}
Under the hypotheses (H), $\mathcal{L}(\lambda)$ and its analytic
continuation admit an integral representation (cf.\ Thm.~6.2 in
\cite{Schilling12}), namely 
\begin{equation}
\mathcal{L}(\lambda)=\int_{(0,\infty)}\frac{\lambda}{\lambda+t}\,d\sigma(t),\label{eq:lambdaKlambda}
\end{equation}
where $\sigma$ is a Borel measure on $[0,\infty)$, such that $\int_{(0,\infty)}(1+t)^{-1}\,d\sigma(t)<\infty$.

Here we give some concrete examples of kernels $k$ and show that
its Laplace transform $\mathcal{K}$ satisfies \eqref{eq:H1} and
\eqref{eq:H2} above. \begin{example}[$\alpha$-Stable subordinator]
\label{exa:alpha-stable1}Let $k$ be the kernel \eqref{eq:kalpha}
corresponding to the Caputo-Djrbashian fractional derivative $\mathbb{D}_{t}^{(\alpha)}$
of order $\alpha\in(0,1)$. Then its Laplace transform is given by
\[
\mathcal{K}(\lambda)=\frac{1}{\Gamma(1-\alpha)}\int_{0}^{\infty}e^{-\lambda t}t^{-\alpha}\,dt=\lambda^{\alpha-1}.
\]
It is easy to verify that \eqref{eq:H1} and \eqref{eq:H2} are satisfied
for $\mathcal{K}$ and $\mathcal{L}$. \end{example}

\begin{example}[Gamma subordinator] \label{exa:gamma-subordinator}Let
$k$ be the kernel defined by 
\[
\mathbb{R}_{+}\ni t\mapsto k(t):=a\Gamma(0,bt),\quad a,b>0,
\]
where $\Gamma(\nu,x):=\int_{x}^{\infty}t^{\nu-1}e^{-t}\,dt$ is the
upper incomplete Gamma function. The Laplace transform of $k$ is
given by 
\[
\mathcal{K}(\lambda)=\frac{a}{\lambda}\log\left(1+\frac{\lambda}{b}\right),\quad\lambda>0.
\]
Again, the properties \eqref{eq:H1} and \eqref{eq:H2} are simple
to verify. \end{example}

\begin{example}[Inverse Gaussian subordinator] Let $a\ge0$ and
$b>0$ be given and define the kernel $k$ by 
\[
\mathbb{R}_{+}\ni t\mapsto k(t):=\sqrt{\frac{b}{2\pi}}\left(\frac{2}{\sqrt{t}}e^{-\frac{at}{2}}-\sqrt{2a\pi}(1-\text{erf}(z))\right),\quad z:=\sqrt{\frac{at}{2}},
\]
where $\text{erf}(z):=\frac{2}{\sqrt{\pi}}\int_{0}^{z}e^{-t^{2}}\,dt$
is the error function. The Laplace transform of $k$ can be computed
and is given by 
\[
\mathcal{K}(\lambda)=\frac{\sqrt{b}}{\lambda}\big(2\sqrt{2\lambda+a}-\sqrt{a}\big),\quad\lambda>0.
\]
The properties \eqref{eq:H1} and \eqref{eq:H2} follows easily. \end{example}

\subsection{Special Classes of Kernels}

Here we collect some classes of kernels $k$ and its Laplace transform
asymptotics since they play a major role in this work. Two classes
are emphasized, the class corresponding to the distributed order derivative
with $k$ given by \eqref{eq:distributed-kernel} and the class of
the general fractional derivative \eqref{eq:general-derivative} for
which $\mathcal{K}$ is a Stieltjes function.

\subsubsection{Distributed order derivatives}

The following proposition refers to the special case of distributed
order derivative, see \cite{Kochubei2008} for the proof. We denote
the negative real axis by $\mathbb{R}_{-}:=(-\infty,0]$. \begin{prop}[{cf.\ \cite[Prop.~2.2]{Kochubei2008}}]
\label{prop:distr-order-prop-K} 
\begin{enumerate}
\item Let $\mu\in C^{2}([0,1])$ be given. If $\lambda\in\mathbb{C}\backslash\mathbb{R}_{-}$
with $|\lambda|\to\infty$, then 
\begin{equation}
\mathcal{K}(\lambda)=\frac{\mu(1)}{\log\lambda}+O\left((\log|\lambda|)^{-2}\right).\label{infty}
\end{equation}
More precisely, if $\mu\in C^{3}([0,1])$, then 
\[
\mathcal{K}(\lambda)=\frac{\mu(1)}{\log\lambda}-\frac{\mu'(1)}{(\log\lambda)^{2}}+O\left((\log|\lambda|)^{-3}\right).
\]
\item Let $\mu\in C([0,1])$ and $\mu(0)\ne0$ be given. If $\lambda\in\mathbb{C}\setminus\mathbb{R}_{-}$,
then 
\begin{equation}
\mathcal{K}(\lambda)\sim\frac{1}{\lambda}\log\left(\frac{1}{\lambda}\right)^{-1}\mu(0),\quad\mathrm{as}\;\lambda\to0.\label{zero}
\end{equation}
\item Let $\mu\in C([0,1])$ be such that $\mu(\alpha)\sim a\alpha^{s}$,
$a>0$, $s>0$. If $\lambda\in\mathbb{C}\setminus\mathbb{R}_{-}$,
then 
\begin{equation}
\mathcal{K}(\lambda)\sim a\Gamma(1+s)\frac{1}{\lambda}\log\left(\frac{1}{\lambda}\right)^{-1-s},\quad\mathrm{as}\;\lambda\to0.\label{eq:zero1}
\end{equation}
\end{enumerate}
\end{prop}

\subsubsection{Classes of Stieltjes functions}

In general if $k\in L_{\mathrm{loc}}^{1}(\mathbb{R}_{+})$, under
the assumption (H), it follows from \eqref{eq:lambdaKlambda} that
the Stieltjes function $\mathcal{K}$ admits the integral representation
\begin{equation}
\mathcal{K}(\lambda)=\int_{(0,\infty)}\frac{1}{\lambda+t}\,d\sigma(t),\quad\lambda>0.\label{eq:LT-k}
\end{equation}
In other words, $\mathcal{K}$ is the Stieltjes transform of the Borel
measure $\sigma$. If $\sigma$ is absolutely continuous with respect
to Lebesgue measure with a continuous density $\varphi$ on $[0,\infty)$,
then $\mathcal{K}$ turns out 
\begin{equation}
\mathcal{K}(\lambda)=\int_{0}^{\infty}\frac{\varphi(t)}{\lambda+t}\,dt.\label{eq:LT-k-1}
\end{equation}
If in addition $\varphi$ has the asymptotic 
\begin{align}
\varphi(t) & \sim Ct^{-\alpha},\quad\mathrm{as}\;t\to\infty,\;0<\alpha<1,\label{wong_infty}\\
\varphi(t) & \sim Ct^{\theta-1},\quad\mathrm{as}\;t\to0,\;0<\theta<1,\label{power_zero}
\end{align}
then, $\varphi\in L_{\mathrm{loc}}^{1}([0,\infty))$ and it follows
from \cite[Thm.~1, page~299]{Wong2001} (see also \cite{FL}) that
the asymptotic (\ref{wong_infty}) implies the asymptotics for $\mathcal{K}$
\begin{equation}
\mathcal{K}(\lambda)\sim C\lambda^{-\alpha},\quad\mathrm{as}\;\lambda\to\infty.\label{wong_infty1}
\end{equation}
For the asymptotic of $\mathcal{K}$ at the origin, we have the following
lemma, see \cite[Lem.~7]{KKS2018}. \begin{lem} \label{lem:GFD-Lemma}Suppose
that 
\begin{equation}
\varphi(t)=Ct^{\theta-1}+\psi(t),\quad0<\theta<1,\label{eq:density}
\end{equation}
where $|\psi(t)|\le Ct^{\theta-1+\delta}$, $0<t\le t_{0}$, and $|\psi(t)|\le Ct^{-\varepsilon}$,
$t>t_{0}$ . Here $0<\delta<1-\theta$ and $\varepsilon>0$. Then
\[
\mathcal{K}(\lambda)\sim C\lambda^{\theta-1},\quad\mathrm{as}\;\lambda\to0.
\]
\end{lem}

The function $[0,\infty)\ni\lambda\mapsto e^{-\tau\lambda\mathcal{K}(\lambda)}$,
$\tau>0$ is the composition of a complete Bernstein and a completely
monotone function, then by Theorem\ 3.7 in \cite{Schilling12} it
is a completely monotone function. By Bernstein's theorem (see \cite[Thm.~1.4]{Schilling12}),
for each $\tau\ge0$, there exists a probability measure $\nu_{\tau}$
on $\mathbb{R}_{+}$ such that 
\begin{equation}
e^{-\tau\lambda\mathcal{K}(\lambda)}=\int_{(0,\infty)}e^{-\lambda s}\,d\nu_{\tau}(s).\label{eq:Laplace-family}
\end{equation}
Define 
\begin{equation}
G_{t}(\tau):=\int_{(0,t)}k(t-s)\,d\nu_{\tau}(s).\label{eq:density-G_ttau}
\end{equation}

The function $G_{t}(\tau)$ is a central object of this paper, therefore
we collect some of its properties, see Lem.~3.1 in \cite{Toaldo2015}.

\begin{enumerate}
\item The $t$-Laplace transform of $G_{t}(\tau)$ is given by 
\begin{equation}
g(\lambda,\tau):=\int_{0}^{\infty}e^{-\lambda t}G_{t}(\tau)\,dt=\mathcal{K}(\lambda)e^{-\tau\lambda\mathcal{K}(\lambda)}.\label{eq:tLaplace-G}
\end{equation}
\item The double $(t,\tau)$-Laplace transform of $G_{t}(\tau)$ is equal
to 
\[
\int_{0}^{\infty}\int_{0}^{\infty}e^{-\lambda t-p\tau}G_{t}(\tau)\,dt\,d\tau=\frac{\mathcal{K}(\lambda)}{\lambda\mathcal{K}(\lambda)+p}.
\]
\item For each fixed $t\in\mathbb{R}_{+}$, $G_{t}(\tau)$ is a probability
density, therefore $\mathbb{R}_{+}\ni\tau\mapsto G_{t}(\tau)\in L^{\infty}(\mathbb{R}_{+})\cap L^{1}(\mathbb{R}_{+})$. 
\end{enumerate}

\subsection{Probabilistic Interpretation}

\label{subsec:Probabilistic-Interpretation}As the map $[0,\infty)\ni\lambda\mapsto\Phi(\lambda):=\lambda\mathcal{K}(\lambda)$
is a complete Bernstein function, then we may define a subordinator
$S$ by its Laplace transform as 
\[
\mathbb{E}[e^{-\lambda S_{t}}]=e^{-t\Phi(\lambda)}=e^{-t\lambda\mathcal{K}(\lambda)},\quad\lambda\ge0,
\]
and $\Phi$ is called the \emph{Laplace exponent} or \emph{cumulant}
of $S$. The associated L{\'e}vy measure $\sigma$ has support in
$[0,\infty)$, fulfils 
\begin{equation}
\int_{(0,\infty)}(1\wedge\tau)\,d\sigma(\tau)<\infty,\label{eq:Levy-condition}
\end{equation}
and the Laplace exponent $\Phi$ is represented by 
\begin{equation}
\Phi(\lambda)=\int_{(0,\infty)}(1-e^{-\lambda\tau})\,d\sigma(\tau).\label{eq:Levy-Khintchine}
\end{equation}
The equality \eqref{eq:Levy-Khintchine} is known as the L{\'e}vy-Khintchine
formula for the subordinator $S$. The kernel $k$ is related to the
subordinator $S$ via the L{\'e}vy measure $\sigma$, namely if we
set 
\[
k(t)=\sigma\big((t,\infty)\big),\quad\forall t\in[0,\infty)
\]
it is easy to compute its Laplace transform. In fact, for any $\lambda\ge0$
\[
\int_{0}^{\infty}e^{-\lambda t}\int_{0}^{t}\,d\sigma(s)\,dt=\int_{0}^{\infty}\int_{0}^{s}e^{-\lambda t}\,dt\,d\sigma(s)=\frac{1}{\lambda}\Phi(\lambda)=\mathcal{K}(\lambda).
\]
Denote by $E$ the inverse process of the subordinator $S$, that
is

\begin{equation}
E_{t}:=\inf\{s\ge0:\;S_{s}\ge t\}=\sup\{s\ge0:\;S_{t}\le s\}.\label{eq:inverse-sub}
\end{equation}
Then the marginal density of $E(t)$ is the function $G_{t}(\tau)$,
$t,\tau\ge0$, more precisely 
\[
G_{t}(\tau)\,d\tau=\partial_{\tau}\mathbf{P}(E_{t}\le\tau)=\partial_{\tau}\mathbf{P}(S_{\tau}\ge t)=-\partial_{\tau}\mathbf{P}(S_{\tau}<t).
\]

\section{Evolution Equations and the General Method}

\label{sec:GFEE}In this section we develop a general method to study
the long time behavior of the subordination by the function $G_{t}(\tau)$
(introduced in \eqref{eq:density-G_ttau}) of the solution $u_{0}(x,t)$
of a Cauchy problem (CP). We choose three of these CPs, namely with
exponential time decay, the heat equation and linear non-local diffusions.

From now on $L$ denotes always a slowly varying function (SVF), that
is 
\[
\lim_{x\to\infty}\frac{L(\lambda x)}{L(x)}=1,\qquad\mathrm{for\;any\;}\lambda>0,
\]
and $C$, $C'$ are constants which change from line to line.

\subsection{The General Method}

\label{subsec:The-General-Method}Let $A$ be a generic (heuristic)
Markov generator defined on functions $u_{0}(x,t)$, $t>0$, $x\in\mathbb{R}^{d}$.
In Subsection~\ref{subsec:Applications-examples} we present concrete
examples of such Markov generators. Consider the evolution equations
of the following type 
\begin{equation}
\begin{cases}
{\displaystyle \frac{\partial u_{0}(x,t)}{\partial t}} & =Au_{0}(x,t)\\
u_{0}(x,0) & =\xi(x),
\end{cases}\label{eq:Evol-eq-1}
\end{equation}
which we assume a solution $u_{0}(x,\cdot)\in L^{1}(\mathbb{R}_{+})$
is known. We are interested in studying the subordination of the solution
$u_{0}(x,t)$ by the density $G_{t}(\tau)$, that is the function
$u(x,t)$ defined by 
\begin{equation}
u(x,t):=\int_{0}^{\infty}u_{0}(x,\tau)G_{t}(\tau)\,d\tau,\quad x\in\mathbb{R}^{d},\;t\ge0.\label{eq:subordination}
\end{equation}
The subordination principle, see \cite{Bazhlekova00}, tells that
$u(x,t)$ is the solution of the general fractional differential equation
\begin{equation}
\begin{cases}
(\mathbb{D}_{t}^{(k)}u)(x,t) & =Au(x,t)\\
u(x,0) & =\xi(x),
\end{cases}\label{eq:Evol-eq-k}
\end{equation}
with the same operator $A$ acting in the spatial variables $x$ and
the same initial condition $\xi$. \begin{rem} \label{rem:Solutions} 
\begin{enumerate}
\item The appropriate notions of the solutions of \eqref{eq:Evol-eq-1}
and \eqref{eq:Evol-eq-k} depend on the specific setting. They were
explained 
\begin{enumerate}
\item in \cite{Kochubei11} for the case where $A$ is the Laplace operator
on $\mathbb{R}^{n}$, 
\item in \cite{Bazhlekova00,Baz01,Bazhlekova2015} with abstract semigroup
generators for special classes of kernels $k$, 
\item in \cite{Pruss12} for abstract Volterra equations. 
\end{enumerate}
\item There is also a probabilistic interpretation of the subordination
identities (see, for example, \cite{Kolokoltsov2011}). In the models
of statistical dynamics we deal with a subordination of measure flows
that will give a weak solution to the corresponding general fractional
equation. 
\end{enumerate}
\end{rem}

In order to study the time evolution of $u(x,t)$ one possibility
is to define its Cesaro mean 
\[
M_{t}\big(u(x,t)\big):=\frac{1}{t}\int_{0}^{t}u(x,s)\,ds
\]
and investigate its long time behavior. Notice that the Cesaro mean
of $u(x,t)$ may be written as 
\begin{align}
M_{t}\big(u(x,t)\big) & =\int_{0}^{\infty}u_{0}(x,\tau)\left(\frac{1}{t}\int_{0}^{t}G_{s}(\tau)\,ds\right)d\tau\nonumber \\
 & =\int_{0}^{\infty}u_{0}(x,\tau)M_{t}\big(G_{t}(\tau)\big)d\tau.\label{eq:Cesaro-mean-u}
\end{align}
Therefore, we are led to investigate the Cesaro mean of the density
$G_{t}(\tau)$ which determine the long time behavior of $u(x,t)$
once the integral in \eqref{eq:Cesaro-mean-u} exists. To this end,
first we introduce a suitable class of admissible $k(t)$, then we
show a theorem which, for each fixed $\tau\in[0,\infty)$, gives a
connection between the Cesaro mean of $G_{t}(\tau)$ and Cesaro mean
of $k(t)$. We assume $u_{0}(x,\cdot)\in L^{1}(\mathbb{R}_{+})$,
then the asymptotic of the integral in \eqref{eq:Cesaro-mean-u} is
a consequence of the pointwise convergence in $\tau$ and a uniform
bound that gives the possibility to apply Lebesgue's dominated convergence
theorem. \begin{defn}[Admissible kernels - $\mathbb{K}(\mathbb{R}_{+})$]
The subset $\mathbb{K}(\mathbb{R}_{+})\subset L_{\mathrm{loc}}^{1}(\mathbb{R}_{+})$
of admissible kernels $k$ is defined by those elements in $L_{\mathrm{loc}}^{1}(\mathbb{R}_{+})$
satisfying (H) such that for some $s_{0}>0$ 
\[
\liminf_{\lambda\to0+}\frac{1}{\mathcal{K}(\lambda)}\int_{0}^{\nicefrac{s_{0}}{\lambda}}k(t)\,dt>0\tag*{(A1)}
\]
and 
\[
\lim_{\genfrac{}{}{0pt}{2}{t,r\to\infty}{\frac{t}{r}\to1}}\left(\int_{0}^{t}k(s)\,ds\right)\left(\int_{0}^{r}k(s)\,ds\right)^{-1}=1.\tag*{(A2)}
\]
\end{defn}

The assumptions (A1) and (A2) are easy to check for the classes we
introduced in Section~\ref{sec:GFD}.

The following theorem establishes an asymptotic relation between the
Cesaro means of the density $G_{t}(\tau)$ and Cesaro mean of $k(t)\in\mathbb{K}(\mathbb{R}_{+})$,
for each fixed $\tau\in[0,\infty)$. \begin{thm} \label{thm:main-result}Let
$\tau\in[0,\infty)$ be fixed and $k\in\mathbb{K}(\mathbb{R}_{+})$
a given admissible kernel. Define the map $G_{\cdot}(\tau):[0,\infty)\longrightarrow\mathbb{R}_{+}$,
$t\mapsto G_{t}(\tau)$ such that $\int_{0}^{\infty}e^{-\lambda t}G_{t}(\tau)\,dt$
exists for all $\lambda>0$. Then 
\[
\lim_{t\to\infty}\left(\int_{0}^{t}G_{s}(\tau)\,ds\right)\left(\int_{0}^{t}k(s)\,ds\right)^{-1}=1
\]
or 
\[
M_{t}\big(G_{t}(\tau)\big)=\frac{1}{t}\int_{0}^{t}G_{s}(\tau)\,ds\sim\frac{1}{t}\int_{0}^{t}k(s)\,ds=M_{t}\big(k(t)\big),\quad t\to\infty
\]
 and $M_{t}\big(G_{t}(\tau)\big)$ is uniformly bounded in $\tau\in\mathbb{R}_{+}$.
\end{thm}

\begin{proof} The $t$-Laplace transform of $G_{t}(\tau)$ exists
for any $\lambda>0$, cf.\ \eqref{eq:tLaplace-G}. Then the result
of the theorem for each $\tau>0$ follows from Corollary 3.3. in \cite{Li2007}
with $X_{+}=\mathbb{R}_{+}$, $G_{t}=u(t)$, $k=g$ and $x=1$. The
uniform bound in $\tau$ follows from the obvious uniform bound $e^{-\tau\lambda\mathcal{K}(\lambda)}\leq1.$
\end{proof}

We have now all the necessary tools to investigate the Cesaro mean
of the density $G_{t}(\tau)$ for all the classes of admissible kernels.
The following three classes of admissible kernels $k\in\mathbb{K}(\mathbb{R}_{+})$
are studied, and they are given in terms of their Laplace transform
$\mathcal{K}(\lambda)$ as $\lambda\to0$ 
\[
\mathcal{K}(\lambda)=\lambda^{\theta-1},\quad0<\theta<1.\tag*{(C1)}
\]
\[
\mathcal{K}(\lambda)\sim\lambda^{-1}L\left(\frac{1}{\lambda}\right),\quad L(x):=\mu(0)\log(x)^{-1}.\tag*{(C2)}
\]
\[
\mathcal{K}(\lambda)\sim\lambda^{-1}L\left(\frac{1}{\lambda}\right),\quad L(x):=C\log(x)^{-1-s},\;s>0,\;C>0.\tag*{(C3)}
\]
To idea to study the Cesaro mean of the density $G_{t}(\tau)$, having
in mind the result of Theorem \ref{thm:main-result}, is to check
the behavior of $\mathcal{K}(\lambda)$ as $\lambda\to0$ and an application
of the Karamata-Tauberian theorem.
\begin{description}
\item [{(C1).}] We have in this case 
\[
\mathcal{K}(\lambda)=\lambda^{\theta-1}=\lambda^{-\rho}L\left(\frac{1}{\lambda}\right),
\]
where $\rho:=1-\theta\ge0$ and $L(x):=1$ is a `trivial' SVF. Then
we obtain as $t\to\infty$ 
\[
\int_{0}^{t}k(s)\,ds\sim Ct^{\rho}L(t)\Leftrightarrow M_{t}(k(t))\sim Ct^{-\theta}.
\]
\item [{(C2).}] We have, as $\lambda\to0$ 
\[
\mathcal{K}(\lambda)\sim\lambda^{-1}\log\left(\frac{1}{\lambda}\right)^{-1}\mu(0)=\lambda^{-1}L\left(\frac{1}{\lambda}\right),\;\mathrm{as}\;\lambda\to0,
\]
where $L(x):=\mu(0)\log(x)^{-1}$ is a SVF. Hence, we have $M_{t}(k(t))\sim C\log(t)^{-1}$,
as $t\to\infty$.
\item [{(C3).}] The Laplace transform for each $s>0$ 
\[
\mathcal{K}(\lambda)\sim C\lambda^{-1}\log\left(\frac{1}{\lambda}\right)^{-1-s}=\lambda^{-1}L\left(\frac{1}{\lambda}\right),\quad\mathrm{as}\;\lambda\to0,
\]
where $L(x):=C\log(x)^{-1-s}$ is a SVF. It follows that $M_{t}(k(t))\sim C\log(t)^{-1-s}$,
as $t\to\infty.$
\end{description}

\subsection{Applications to Concrete Examples}

\label{subsec:Applications-examples}

\subsubsection{Exponential decay}

\label{subsec:Exponential-decay}Let us assume that the solution $u_{0}(x,t)$
of the Cauchy problem \eqref{eq:Evol-eq-1} is such that 
\begin{equation}
\sup_{x\in\mathbb{R}^{d}}|u_{0}(x,t)|\le Ce^{-\gamma t},\quad\gamma>0.\label{eq:exp-decay}
\end{equation}
This behavior of $u_{0}$ may be justified in a number of cases of
PDEs. We derive the long time behavior of the subordination $u(x,t)$
defined in \eqref{eq:subordination} using the general method above.
As the function $\mathbb{R}_{+}\ni t\mapsto u_{0}(x,t)\in\mathbb{R}_{+}$
is integrable, then the long time behavior of the Cesaro mean of $u(x,t)$
reduces to the study of the Cesaro mean of the admissible kernel $k(t)$.
We derive the long time behavior of the Cesaro mean of $k(t)$ through
its Laplace transform $\mathcal{K}(\lambda)$ by an application of
the Karamata-Tauberian theorem. 
\begin{description}
\item [{(C1).}] For the first class of kernels (C1) it is easy to see that
the Cesaro mean of $k$ is given, as before, by 
\begin{equation}
M_{t}(u(x,\cdot))\sim Ct^{-\theta},\;t\to\infty.\label{eq:Cesaro-mean-C1}
\end{equation}
\item [{(C2).}] For the class (C2), we obtain 
\begin{equation}
M_{t}(u(x,\cdot))\sim C\log(t)^{-1},\;t\to\infty.\label{eq:Cesaro-mean-C2}
\end{equation}
\item [{(C3).}] Now we look at class (C3) which gives 
\begin{equation}
M_{t}(u(x,\cdot))\sim C\log(t)^{-1-s},\;t\to\infty.\label{eq:Cesaro-mean-C3}
\end{equation}
\end{description}

\subsubsection{The Heat Equation}

We consider the Cauchy problem given by 
\begin{equation}
\begin{cases}
{\displaystyle \frac{\partial u_{0}(x,t)}{\partial t}} & =\Delta u_{0}(x,t)\\
u_{0}(x,0) & =\varphi(x),
\end{cases}\label{eq:Cauchy-problem}
\end{equation}
where $\varphi\in L^{1}(\mathbb{R}^{d})$. If $\mathcal{G}_{t}(x)$
denotes the fundamental solution (also known as Green function) of
the Cauchy problem \eqref{eq:Cauchy-problem}, then the solution $u_{0}(x,t)$
is written as a convolution between the initial condition $\varphi$
and $\mathcal{G}_{t}$, that is 
\[
u_{0}(x,t)=(\varphi*\mathcal{G}_{t})(x).
\]
Using the Young convolution inequality $\|u_{0}(\cdot,t)\|_{\infty}\le\|\varphi\|_{L^{1}}\|\mathcal{G}_{t}\|_{\infty}$,
the solution $u_{0}(x,t)$ is continuous in $t$ and bounded in $x$
in the supremum norm. In addition, it is not difficult to see that
$u_{0}(x,t)$ satisfies 
\begin{equation}
\sup_{x\in\mathbb{R}^{d}}|u_{0}(x,\tau)|\le C,\;\tau\in[0,1]\label{eq:Assump-uzero-1}
\end{equation}
and 
\begin{equation}
\sup_{x\in\mathbb{R}^{d}}|u_{0}(x,\tau)|\le\frac{C}{\tau^{d/2}},\;\tau\in]1,\infty).\label{eq:Assump-uzero-2}
\end{equation}

The function $u(x,t)$ is defined as the subordination of $u_{0}(x,t)$
by the density $G_{t}(\tau)$, see \eqref{eq:subordination}. 

As $u_{0}(x,t)$ is bounded in a neighbourhood of $\tau=0+$, then
the only important contribution for the long time behavior of $u(x,t)$
comes from $\tau>1$. On the other hand, the map $[1,\infty)\ni\tau\mapsto\frac{1}{\tau^{d/2}}\in\mathbb{R}_{+}$
belongs to $L^{1}(\mathbb{R}_{+})$ for $d\ge3$. Therefore using
the results from Subsection~\ref{subsec:The-General-Method} we may
derive the long time behavior of the Cesaro mean of $u(x,t)$ as in
the previous example for each classes (C1), (C2), and (C3). See \eqref{eq:Cesaro-mean-C1},
\eqref{eq:Cesaro-mean-C2} and \eqref{eq:Cesaro-mean-C3}. Notice
that for $d=1$ and $d=2$ this method does not allow us to take any
conclusion on the long time behavior of the Cesaro mean of $u(x,t)$
since $\frac{1}{\tau^{d/2}}\notin L^{1}(\mathbb{R}_{+})$. On Section~\ref{sec:Alternative-Methods}
we use an alternative method which allow us to do so.

\subsubsection{Linear Non-local Diffusion}

\label{subsec:LNonlocal-Diffusion}We consider the linear non-local
diffusion, see for instance \cite[Ch.~1]{Rossi2010} 
\begin{equation}
\begin{cases}
{\displaystyle \frac{\partial u_{0}(x,t)}{\partial t}} & =a*u_{0}(x,t)-u_{0}(x,t)={\displaystyle \int_{\mathbb{R}^{d}}a(x-y)(u_{0}(y,t)\,dy-u_{0}(x,t)}\\
u_{0}(x,0) & =\varphi(x),
\end{cases}\label{eq:LNonlocalDiff}
\end{equation}
for $x\in\mathbb{R}^{d}$, $t>0$, and $a\in C(\mathbb{R}^{d},\mathbb{R})$
is a radial density function, that is a nonnegative radial function
with $a(0)>0$ and $\langle a\rangle:=\int_{\mathbb{R}^{d}}a(x)\,dx=1.$
The notion of a solution of \eqref{eq:LNonlocalDiff} is a function
$u_{0}\in C(\mathbb{R}_{+},L^{1}(\mathbb{R}^{d}))$ such that \eqref{eq:LNonlocalDiff}
is satisfied in the integral sense 
\[
u_{0}(x,t)=\varphi(x)+\int_{0}^{t}\int_{\mathbb{R}^{d}}a(x-y)u_{0}(y,s)\,dy-u_{0}(x,s)\,ds.
\]
The existence and uniqueness of solutions of the CP \eqref{eq:LNonlocalDiff}
may be shown using the Fourier transform technic. In the sequel, $\hat{f}$
denotes de Fourier transform of $f\in L^{1}(\mathbb{R}^{d})$ defined
by 
\[
\hat{f}(\xi):=\int_{\mathbb{R}^{d}}e^{-i\langle x,\xi\rangle}f(x)\,dx,
\]
where $\langle\cdot,\cdot\rangle$ denotes de scalar product in $\mathbb{R}^{d}$.
The following theorem states under which conditions on $a$ and $\varphi$
the CP \eqref{eq:LNonlocalDiff} has a unique solution, see Theorem~1.3
in \cite[Ch.~1]{Rossi2010} for more details and other properties
of the solution $u_{0}(x,t)$. Here we emphasize the uniform bound
of $u_{0}(x,t)$ in $x$ as the most relevant for our considerations
below. \begin{thm} Assume that there exist $A>0$ and $0<r\le2$
such that 
\[
\hat{a}(\xi)=1-A|\xi|^{r}+o(|\xi|^{r})\quad\mathrm{as}\;\xi\to0.
\]
For any nonnegative $\varphi$ such that $\varphi,\hat{\varphi}\in L^{1}(\mathbb{R}^{d})$,
there exits a unique solution $u_{0}(x,t)$ of the CP \eqref{eq:LNonlocalDiff}
such that 
\[
\|u_{0}(\cdot,t)\|_{L^{\infty}(\mathbb{R}^{d})}\le Ct^{-d/r}.
\]
\end{thm}

\begin{rem} As the solution $u_{0}(x,t)$ is time continuous and
uniformly bounded in $x$, then it is easy to derive the following
properties of $u_{0}(x,t)$ \end{rem}

\begin{equation}
\sup_{x\in\mathbb{R}^{d}}|u_{0}(x,\tau)|\le C,\;\tau\in[0,1],\label{eq:estimate-uzero-LNLD1}
\end{equation}

\begin{equation}
\sup_{x\in\mathbb{R}^{d}}|u_{0}(x,\tau)|\le C\tau^{-d/r},\;\tau\in]1,\infty).\label{eq:estimate-uzero-LNLD2}
\end{equation}

Our aim now is to study the function $u(x,t)$ given by the subordination
of $u_{0}(x,t)$ by the density $G_{t}(\tau)$ as in \eqref{eq:subordination},
that is determine the long time behavior of $u(x,t)$ for all the
classes of admissible kernels $k\in\mathbb{K}(\mathbb{R}_{+})$.

For $d\ge3$ the function $\mathbb{R}_{+}\ni\tau\mapsto\tau^{-d/r}\in\mathbb{R}_{+}$
is integrable, therefore the long time behavior of $M_{t}(u(x,t))$
reduces to that of $M_{t}(G_{t}(\tau))$. For the three classes of
admissible kernels $k\in\mathbb{K}(\mathbb{R}_{+})$, they are given
by \eqref{eq:Cesaro-mean-C1}, \eqref{eq:Cesaro-mean-C2}, \eqref{eq:Cesaro-mean-C3}.

\section{Alternative Method for Subordinated Dynamics}

\label{sec:Alternative-Methods}In this section we investigate the
long time behavior of the subordination dynamics $u(x,t)$ for the
three CP problems from Subsection~\ref{subsec:Applications-examples}
using an alternative method, the Laplace transform. The possibility
to apply this alternative method is related to the a priori information
of the initial solution $u_{0}(x,t)$. Here we would like to emphasize
the results obtained for the heat equation and the linear non-local
diffusion. More precisely, the general method from Section~\ref{sec:GFEE}
does not allow us to obtain the long time behavior of $M_{t}(u(x,t))$
for these examples if the dimension $d=1$ and $d=2$, while the Laplace
transform method does for any dimension $d\ge1$.

\subsection{Exponential decay}

We have the exponential decay of the initial solution $u_{0}(x,t)$,
see \eqref{eq:exp-decay}. Computing the $t$-Laplace transform of
$u(x,t)$ and using \eqref{eq:tLaplace-G} to obtain 
\[
(\mathscr{L}u(x,\cdot))(\lambda)=C\frac{\mathcal{K}(\lambda)}{\lambda\mathcal{K}(\lambda)+\gamma}.
\]

We investigate each class of admissible kernels $k\in\mathbb{K}(\mathbb{R}_{+})$,
that is (C1), (C2) and (C3). 
\begin{description}
\item [{(C1).}] It follows that 
\[
(\mathscr{L}u(x,\cdot))(\lambda)=C\frac{\lambda^{\theta-1}}{\lambda^{\theta}+\gamma}=\lambda^{-(1-\theta)}L\left(\frac{1}{\lambda}\right),\quad L(x):=\frac{C}{x^{-\theta}+\gamma}.
\]
Then the Karamata-Tauberian theorem gives 
\[
M_{t}(u(x,t))\sim Ct^{-\theta}\frac{1}{t^{-\theta}+\gamma}\sim Ct^{-\theta},\;t\to\infty.
\]
\item [{(C2).}] We have, as $\lambda\to0$ 
\[
(\mathscr{L}u(x,\cdot))(\lambda)\sim C\lambda^{-1}L\left(\frac{1}{\lambda}\right),\quad L(x):=C\frac{(\log(x))^{-1}}{(\log(x))^{-1}+\gamma}.
\]
And again, an application of the Karamata-Tauberian theorem yields
\[
M_{t}(u(x,t))\sim C\log(t)^{-1}\frac{1}{(\log(t))^{-1}+\gamma}\sim C\log(t)^{-1},\quad t\to\infty.
\]
\item [{(C3).}] For that class one obtains 
\[
(\mathscr{L}u(x,\cdot))(\lambda)\sim\lambda^{-1}L\left(\frac{1}{\lambda}\right),\quad L(x):=C\frac{(\log(x))^{-1-s}}{(\log(x))^{-1-s}+\gamma}.
\]
By the Karamata-Tauberian theorem we have 
\[
M_{t}(u(x,t))\sim C\log(t)^{-1-s}\frac{1}{(\log(t))^{-1-s}+\gamma}\sim C\log(t)^{-1-s},\;t\to\infty.
\]
\end{description}
In conclusion, this alternative method reproduces the same type of
decay of the Cesaro mean of $u(x,t)$ as the general method from Section~\ref{sec:GFEE}
for this example.

\subsection{The Heat Equation}

We compute the $t$-Laplace transform of $u(x,t)$ and then apply
the Karamata-Tauberian theorem. We have, using again \eqref{eq:tLaplace-G},
that 
\[
(\mathscr{L}u(x,\cdot))(\lambda)=\mathcal{K}(\lambda)\int_{0}^{\infty}u_{0}(x,\tau)e^{-\tau\lambda\mathcal{K}(\lambda)}\,d\tau.
\]
It follows from \eqref{eq:Assump-uzero-1} and \eqref{eq:Assump-uzero-2}
that the solution $u_{0}(x,\tau)$ is bounded in a neighborhood of
$\tau=0+$, hence the long time behavior of $M_{t}(u(x,t))$ is only
influenced as $\tau>1$, that is the factor 
\[
CK(\lambda)\int_{1}^{\infty}\tau^{-d/2}e^{-\tau\lambda\mathcal{K}(\lambda)}\,d\tau.
\]
The integral on the right-hand side is computed using the upper incomplete
Gamma function 
\begin{equation}
\int_{b}^{\infty}\tau^{\nu}e^{-\tau x}\,d\tau=x^{-\nu-1}\Gamma(\nu+1,bx),\quad\Re(x)>0.\label{eq:IGF}
\end{equation}
Hence, neglecting the constant for $\tau\in[0,1]$, the $t$-Laplace
transform of $u(x,t)$ has the form 
\[
(\mathscr{L}u(x,\cdot))(\lambda)=C\mathcal{K}(\lambda)(\lambda\mathcal{K}(\lambda))^{d/2-1}\Gamma(1-d/2,\lambda\mathcal{K}(\lambda)).
\]
Now we study each class of admissible kernels $k$ satisfying (C1),
(C2) and (C3). Once more the result in each case follows by an application
of the Karamata-Tauberian theorem.
\begin{description}
\item [{(C1).}] We have $\mathcal{K}(\lambda)=\lambda^{\theta-1}$ and
we distinguish the following cases: 
\begin{enumerate}
\item For $d=1$, as $\lambda\to0$ 
\[
(\mathscr{L}u(x,\cdot))(\lambda)=C\lambda^{-(1-\theta/2)}\Gamma(1/2,\lambda^{\theta})=\lambda^{-\rho}L\left(\frac{1}{\lambda}\right),
\]
where $\rho=1-\theta/2$ and $L(x):=C\Gamma(1/2,x^{-\theta})$ is
a SVF. In fact, to see that $L(x)$ is a SVF first we use the relation
\begin{equation}
\Gamma(s,x)=\Gamma(s)-\gamma(s,x),\quad s\neq0,-1,-2,\ldots,\label{eq:IGF-positive}
\end{equation}
where $\gamma(s,x)$ is the lower incomplete Gamma function, the fact
that $x^{-\theta}\to0$ when $x\to\infty$ together with 
\begin{equation}
\gamma(s,x)\sim\frac{x^{s}}{s},\quad x\to0.\label{eq:asym-ILG-zero}
\end{equation}
Hence, by the Karamata-Tauberian theorem the Cesaro mean of $u(x,t)$
behaves as 
\[
M_{t}(u(x,t))\sim Ct^{-\theta/2}L(t)\sim Ct^{-\theta/2},\;t\to\infty.
\]
\item For $d=2$, as $\lambda\to0$ 
\[
(\mathscr{L}u(x,\cdot))(\lambda)\sim\lambda^{-(1-\theta)}L\left(\frac{1}{\lambda}\right),
\]
where $L(x):=C\Gamma(0,x^{-\theta})=CE_{1}(x^{-\theta})$ and $E_{1}(x)$,
$x>0$ is the exponential integral, see \cite[Eq.~(5.1.1)]{AS92}.
For $x\to0$ we have, cf. \cite[Eq.~(5.1.11)]{AS92} 
\begin{equation}
E_{1}(x)\sim-\kappa-\ln(x),\label{eq:asym-Ei-zero}
\end{equation}
where $\kappa$ is the Euler-Mascheroni constant. Then it is simple
to show that $L(x)=CE_{1}(x^{-\theta})$ is a SVF. Thus, by the Karamata-Tauberian
theorem we obtain 
\begin{equation}
M_{t}(u(x,t))\sim Ct^{-\theta}L(t)\sim Ct^{-\theta}\big(\kappa+\log(t^{-\theta})\big),\;t\to\infty.\label{eq:Cesaro-HE-LTA-d=00003D00003D2}
\end{equation}
\item For $d\ge3$, as $\lambda\to0$ 
\[
(\mathscr{L}u(x,\cdot))(\lambda)\sim\lambda^{-(1-\theta)}L\left(\frac{1}{\lambda}\right),
\]
where $L(x):=x^{\theta(1-d/2)}\Gamma(1-d/2,x^{-\theta}).$ To show
that $L(x)$ is a SVF use the relation 
\begin{equation}
\Gamma(s,x)\sim-\frac{x^{s}}{s},\quad\Re(s)<0,\;x\to0.\label{eq:asym-IGF-zero}
\end{equation}
Once more, the Karamata-Tauberian theorem gives 
\[
M_{t}(u(x,t))\sim Ct^{-\theta}L(t)\sim Ct^{-\theta},\;t\to\infty.
\]
\end{enumerate}
\item [{(C2).}] The Laplace transform $\mathcal{K}(\lambda)$ behaves as
$\lambda\to0$ 
\[
\mathcal{K}(\lambda)\sim\lambda^{-1}L\left(\frac{1}{\lambda}\right),\qquad L(x):=\mu(0)\log(x)^{-1}.
\]
We distinguish the cases $d=1$, $d=2$ and $d\ge3$. 
\begin{enumerate}
\item For $d=1$ as $\lambda\to0$ 
\[
(\mathscr{L}u(x,\cdot))(\lambda)\sim\lambda^{-1}L\left(\frac{1}{\lambda}\right),
\]
where $L(x):=C\log(x)^{-1/2}\Gamma(1/2,\mu(0)\log(x)^{-1})$. To verify
that $L(x)$ is a SVF notice that $L(x)$ is the product of two SVF,
then $L(x)$ is SVF, see \cite[Prop.~1.3.6]{Bingham1987}. Hence,
by the Karamata-Tauberian theorem and \eqref{eq:IGF-positive} the
Cesaro mean of $u(x,t)$ is 
\begin{align*}
M_{t}(u(x,t)) & \sim C\log(t)^{-1}\left(\log(t)^{1/2}\Gamma(1/2,\mu(0)\log(t)^{-1})\right)\\
 & \sim C\log(t)^{-1}+C'\log(t)^{-1/2},\;t\to\infty.
\end{align*}
\item For $d=2$ as $\lambda\to0$ 
\[
(\mathscr{L}u(x,\cdot))(\lambda)\sim\lambda^{-1}L\left(\frac{1}{\lambda}\right),
\]
where $L(x):=\mu(0)\log(x)^{-1}E_{1}(\mu(0)\log(x)^{-1})$. Again,
$L(x)$ is a SVF because it is the product of two SVF. Then an application
of the Karamata-Tauberian theorem and \eqref{eq:asym-Ei-zero} yields
\begin{align*}
M_{t}(u(x,t)) & \sim C\log(t)^{-1}E_{1}(\mu(0)\log(t)^{-1})\\
 & \sim C\log(t)^{-1}\big[\kappa+\log\big(\mu(0)\log(t)^{-1}\big)\big],\;t\to\infty.
\end{align*}
\item In general, for any $d\ge3$ as $\lambda\to0$ we have 
\[
(\mathscr{L}u(x,\cdot))(\lambda)\sim\lambda^{-1}L\left(\frac{1}{\lambda}\right),
\]
where $L(x):=\left(\mu(0)\log(x)^{-1}\right)^{d/2}\Gamma(1-d/2,\mu(0)\log(x)^{-1}).$
It is clear that $L(x)$ is a SVF, hence the Karamata-Tauberian theorem
together with \eqref{eq:asym-IGF-zero} implies the long time behavior
for $M_{t}(u(x,\cdot))$, namely 
\begin{align*}
M_{t}(u(x,t)) & \sim C\log(t)^{-1}\left(\log(t)^{1-d/2}\Gamma(1-d/2,\mu(0)\log(t)^{-1})\right)\\
 & \sim C\log(t)^{-1}.
\end{align*}
\end{enumerate}
\item [{(C3).}] Finally, let us investigate the Cesaro mean of $u(x,t)$
for the class (C3), that is where $\mathcal{K}(\lambda)$ behaves
as $\lambda\to0$ 
\[
\mathcal{K}(\lambda)\sim\lambda^{-1}L\left(\frac{1}{\lambda}\right),\quad L(x):=C(\log(x))^{-1-s},\;s>0,\;C>0.
\]
Proceeding as before we distinguish the following cases: 
\begin{enumerate}
\item For $d=1$, as $\lambda\to0$ we have 
\[
(\mathscr{L}u(x,\cdot))(\lambda)\sim\lambda^{-1}L\left(\frac{1}{\lambda}\right),
\]
where 
\begin{align*}
L(x) & :=C\log(t)^{-1-s}\Gamma(1/2,C\log(x)^{-1-s})\\
 & =C\log(t)^{-1-s}\big(\sqrt{\pi}-\gamma(1/2,\log(x)^{-1-s}\big)
\end{align*}
is a SVF since it is the product of two SVF. Then, the Karamata-Tauberian
theorem yields 
\[
M_{t}(u(x,t))\sim C\log(t)^{-1-s}\left(\sqrt{\pi}-2\log(t)^{(-1-s)/2}\right),\;t\to\infty.
\]
\item For $d=2$, as $\lambda\to0$ we have 
\[
(\mathscr{L}u(x,\cdot))(\lambda)\sim\lambda^{-1}L\left(\frac{1}{\lambda}\right),
\]
where $L(x):=C\log(t)^{-1-s}E_{1}(C\log(x)^{-1-s}).$ Then it follows
from Karamata-Tauberian theorem and \eqref{eq:asym-Ei-zero} that
\[
M_{t}(u(x,t))\sim C\log(t)^{-1-s}[\kappa+\log(C\log(t)^{-1-s})],\;t\to\infty.
\]
\item for $d\ge3$, as $\lambda\to0$ 
\[
(\mathscr{L}u(x,\cdot))(\lambda)\sim\lambda^{-1}L\left(\frac{1}{\lambda}\right),
\]
where $L(x):=C\log(x)^{-2-s+d/2}\Gamma(1-d/2,C\log(x)^{-1-s}).$ Again,
$L(x)$ is a SVF as it is a product of two SVF. Then by the Karamata-Tauberian
theorem and \eqref{eq:asym-IGF-zero} we obtain 
\begin{align*}
M_{t}(u(x,t)) & \sim C\log(t)^{-1-s}\big[\log(t)^{d/2-1}\Gamma(1-d/2,C\log(t)^{-1-s})\big]\\
 & \sim C\log(t)^{-1-s}.
\end{align*}
\end{enumerate}
\end{description}
\begin{rem} As a conclusion, the alternative method produces the
same long time decay of the Cesaro mean of $u(x,t)$ compared to the
general method from Section~\ref{sec:GFEE} for $d\ge3$. In addition,
with the Laplace transform method we can handle the dimensions $d=1$
and $d=2$ which was not possible with the general method. \end{rem}

\subsection{Linear Non-local Diffusion}

It follows from \eqref{eq:estimate-uzero-LNLD1} and \eqref{eq:estimate-uzero-LNLD2}
that the solution $u_{0}(x,\tau)$ is bounded in a neighbourhood of
$\tau=0+$, therefore the long time behavior of $M_{t}(u(x,t))$ depends
only on $\tau>1$, that is the factor 
\[
CK(\lambda)\int_{1}^{\infty}\tau^{-d/r}e^{-\tau\lambda\mathcal{K}(\lambda)}\,d\tau.
\]
The integral on the right hand side above is computed using \eqref{eq:IGF}
such that (neglecting a constant) 
\[
(\mathscr{L}u(x,\cdot))(\lambda)=C\mathcal{K}(\lambda)(\lambda\mathcal{K}(\lambda))^{d/r-1}\Gamma(1-d/r,\lambda\mathcal{K}(\lambda)).
\]
We investigate the long time behavior of $M_{t}(u(x,t))$ for the
three classes of admissible kernels (C1) , (C2) and (C3). The analysis
below is similar to the analysis of the heat equation assuming $1<r\le2$. 
\begin{description}
\item [{(C1).}] We have $\mathcal{K}(\lambda)=\lambda^{\theta-1}$, $0<\theta<1$
and 
\[
(\mathscr{L}u(x,\cdot))(\lambda)=\lambda^{-(1-\theta d/r)}L\left(\frac{1}{\lambda}\right),
\]
where $L(x)=C\Gamma(1-d/r,x^{-\theta})$ is a SVF. 
\begin{enumerate}
\item For $d=1$ it follows that 
\[
(\mathscr{L}u(x,\cdot))(\lambda)=\lambda^{-(1-\theta/r)}\Gamma(1-1/r,\lambda^{-\theta})
\]
with $1-\theta/r>0$ and $1-1/r\in(0,1/2]$. As $\Gamma(1-1/r,\lambda^{-\theta})$
is a SVF, then the Karamata-Tauberian theorem gives 
\[
M_{t}(u(x,t))\sim C\lambda^{-\theta/r}\Gamma(1-1/r,\lambda^{-\theta})
\]
and using the equality \eqref{eq:IGF-positive} we obtain 
\[
M_{t}(u(x,t))\sim Ct^{-\theta/r}L(t)\sim Ct^{-\theta/r}.
\]
\item For $d=2$ we have as $\lambda\to0$
\[
(\mathscr{L}u(x,\cdot))(\lambda)\sim\lambda^{-(1-2\theta/r)}\Gamma(1-2/r,\lambda^{-\theta})
\]
such that to have $1-2\theta/r>0$ implies that $r=2$. This case
is similar to the heat equation, see \eqref{eq:Cesaro-HE-LTA-d=00003D00003D2}.
Thus, we have 
\[
M_{t}(u(x,t))\sim Ct^{-\theta}\big(\kappa+\log(t^{-\theta})\big),\quad t\to\infty.
\]
\item For $d\in[3,r/\theta\vee3)$, we have 
\[
(\mathscr{L}u(x,\cdot))(\lambda)\sim\lambda^{-(1-\theta)}L\left(\frac{1}{\lambda}\right),\quad\lambda\to0,
\]
where $L(x)=x^{\theta(1-d/r)}\Gamma(1-d/r,x^{-\theta})$ is a SVF
using \eqref{eq:asym-IGF-zero}. Therefore, we derive the long time
behavior of $M_{t}(u(x,t))$ as a consequence of the Karamata-Tauberian
theorem, namely 
\[
M_{t}(u(x,t))\sim Ct^{-\theta}L(t)\sim Ct^{-\theta},\quad t\to\infty.
\]
\end{enumerate}
\item [{(C2).}] That is the case when $\mathcal{K}(\lambda)\sim\lambda^{-1}L\left(\frac{1}{\lambda}\right)$,
$L(x):=\mu(0)\log(x)^{-1}$ which implies, as $\lambda\to0$ 
\[
(\mathscr{L}u(x,\cdot))(\lambda)\sim C\lambda^{-1}L\left(\frac{1}{\lambda}\right)^{d/r}\Gamma\left(1-d/r,L\left(\frac{1}{\lambda}\right)\right).
\]

\begin{enumerate}
\item For $d=1$ as $\lambda\to0$, we have 
\[
(\mathscr{L}u(x,\cdot))(\lambda)\sim\lambda^{-1}L\left(\frac{1}{\lambda}\right),
\]
where $L(x)=C\log(x)^{-1/r}\Gamma(1-1/r,\mu(0)\log(x)^{-1})$ is a
SVF. Then the Karamata-Tauberian theorem and \eqref{eq:IGF-positive}
yields 
\begin{align*}
M_{t}(u(x,t)) & \sim C\log(t)^{-1}\log(t)^{-1-1/r}\Gamma(1-1/r,\mu(0)\log(x)^{-1})\\
 & \sim C\log(t)^{-1}\big(\Gamma(1-1/r)\log(t)^{-1-1/r}-C'\big),\quad t\to\infty.
\end{align*}
\item Now for $d=2$ we have, as $\lambda\to0$ 
\[
\mathscr{L}(u(x,\cdot))(\lambda)\sim\lambda^{-1}L\left(\frac{1}{\lambda}\right),
\]
where $L(x)=C\log(x)^{-2/r}\Gamma(1-2/r,\mu(0)\log(x)^{-1})$ is a
SVF. 
\begin{enumerate}
\item For the special case $r=2$ it reduces to 
\[
L(x)=C\log(x)^{-1}E_{1}(\mu(0)\log(x)^{-1}).
\]
Then an application of the Karamata-Tauberian theorem and \eqref{eq:asym-Ei-zero}
yields 
\begin{align*}
M_{t}(u(x,t)) & \sim C\log(t)^{-1}E_{1}(\mu(0)\log(t)^{-1})\\
 & \sim C\log(t)^{-1}\big[\kappa+\log(\mu(0)\log(t)^{-1})\big],\quad\text{t\ensuremath{\to\infty.}}
\end{align*}
\item For $1<r<2$, then $-1<1-2/r<0$ and by \eqref{eq:asym-IGF-zero}
\begin{align*}
M_{t}(u(x,t)) & \sim C\log(t)^{-1}\log(x)^{1-2/r}\Gamma(1-2/r,\mu(0)\log(x)^{-1})\\
 & \sim C\log(t)^{-1},\quad\text{t\ensuremath{\to\infty.}}
\end{align*}
\end{enumerate}
\item For $d\ge3$, we obtain as $\lambda\to0$ 
\[
(\mathscr{L}u(x,\cdot))(\lambda)\sim\lambda^{-1}L\left(\frac{1}{\lambda}\right),
\]
where $L(x)=C\log(x)^{-d/r}\Gamma(1-d/r,\mu(0)\log(x)^{-1})$ is a
SVF. As $1-d/r<0$, then by the Karamata-Tauberian theorem and \eqref{eq:asym-IGF-zero}
follows 
\begin{align*}
M_{t}(u(x,t)) & \sim C\log(t)^{-1}\log(x)^{1-d/r}\Gamma(1-d/r,\mu(0)\log(x)^{-1})\\
 & \sim C\log(t)^{-1},\quad\text{t\ensuremath{\to\infty.}}
\end{align*}
\end{enumerate}
\item [{(C3).}] The third class of admissible kernels has Laplace transform
\[
\mathcal{K}(\lambda)\sim\lambda^{-1}L\left(\frac{1}{\lambda}\right),\quad L(x):=C(\log(x))^{-1-s},\;s>0,\;C>0
\]
such that 
\[
(\mathscr{L}u(x,\cdot))(\lambda)\sim C\lambda^{-1}L\left(\frac{1}{\lambda}\right)^{d/r}\Gamma\left(1-d/r,L\left(\frac{1}{\lambda}\right)\right).
\]

\begin{enumerate}
\item First we take $d=1$ and obtain 
\[
(\mathscr{L}u(x,\cdot))(\lambda)\sim\lambda^{-1}L\left(\frac{1}{\lambda}\right),
\]
where $L(x)=C\log(x)^{-(1+s)/r}\Gamma(1-1/r,C\log(x)^{-1-s})$ is
a SVG. Then by the Karamata-Tauberian theorem and \eqref{eq:IGF-positive}
it follows as $t\to\infty$ that 
\begin{align*}
M_{t}(u(x,t)) & \sim C\log(t)^{-1-s}\log(t)^{1+s-(1+s)/r}\Gamma(1-1/r,C\log(t)^{-1-s})\\
 & \sim C\log(t)^{-1-s}\left(\log(t)^{1+s-(1+s)/r}\Gamma(1-1/r)+C'\right).
\end{align*}
\item For $d=2$ we have 
\[
(\mathscr{L}u(x,\cdot))(\lambda)\sim\lambda^{-1}L\left(\frac{1}{\lambda}\right),
\]
where $L(x)=C\log(x)^{-2(1+s)/r}\Gamma\left(1-2/r,C\log(x)^{-1-s}\right)$
is a SVF. 
\begin{enumerate}
\item For $r=2$ the SVF $L(x)$ reduces to 
\[
L(x)=C\log(x)^{-(1+s)}E_{1}\left(C\log(x)^{-1-s}\right)
\]
and then using \eqref{eq:asym-Ei-zero} we obtain 
\[
M_{t}(u(x,t))\sim C\log(t)^{-1-s}(\kappa+\log(C\log(t)^{-1-s})),\quad t\ensuremath{\to\infty.}
\]
\item For $1<r<2$ we have $-1<1-2/r<0$ and 
\begin{align*}
M_{t}(u(x,t)) & \sim\frac{C}{\log(t)^{1+s}}\log(t)^{(1+s)(1-2/r)}\Gamma\left(1-2/r,C\log(t)^{-1-s}\right)\\
 & \sim C\log(t)^{-1-s},\quad t\to\infty.
\end{align*}
\end{enumerate}
\item Finally for $d\ge3$ we have 
\[
(\mathscr{L}u(x,\cdot))(\lambda)\sim\lambda^{-1}L\left(\frac{1}{\lambda}\right),
\]
where $L(x)=C\log(x)^{-d(1+s)/r}\Gamma\left(1-d/r,C\log(x)^{-1-s}\right)$
is a SVF. As before, we obtain 
\begin{align*}
M_{t}(u(x,t)) & \sim\frac{C}{\log(t)^{1+s}}\log(t)^{(1+s)(1-d/r)}\Gamma\left(1-d/r,C\log(t)^{-1-s}\right)\\
 & \sim C\log(t)^{-1-s},\quad t\to\infty.
\end{align*}
\end{enumerate}
\end{description}
In conclusion, both methods produces the same type of long time behavior
for $d\ge3$, in addition for $d=1$ and $d=2$ we are also able to
obtain a decay using this alternative Laplace transform method.

\subsection*{Founding}

Jos{\'e} L. da Silva is a member of the Centro de Investiga{\c c\~a}o
em Matem{\'a}tica e Aplica{\c c\~o}es (CIMA), Universidade da Madeira,
a research centre supported with Portuguese funds by FCT (Funda{\c c\~a}o
para a Ci{\^e}ncia e a Tecnologia, Portugal) through the Project
UID/MAT/04674/2019. Financial support from Bielefeld Graduate School
in Theoretical Sciences, the IRTG 2235 and the CRC 1283 are grateful
acknowledged.

\end{document}